%% file: SimplyConnect.tex
\documentclass[11pt]{amsart}
\usepackage{amssymb,amsthm,amsmath}
\RequirePackage[dvipsnames,usenames]{xcolor}
\usepackage{hyperref}
\usepackage{mathtools,etoolbox}
\usepackage[all]{xy}
\usepackage{tikz}
\usepackage{enumitem}
\usepackage[english]{babel}
\usepackage{graphics}
\usepackage{color}

\hypersetup{
bookmarksdepth=3,
bookmarksopen,
bookmarksnumbered,
pdfstartview=FitH,
colorlinks,
linkcolor=Sepia,
anchorcolor=BurntOrange,
citecolor=MidnightBlue,
citecolor=OliveGreen,
filecolor=BlueViolet,
menucolor=Yellow,
urlcolor=OliveGreen
}

\newcommand{\factor}[2]{\left. \raise 2pt\hbox{\ensuremath{#1}} \right/
        \hskip -2pt\raise -2pt\hbox{\ensuremath{#2}}}

\makeatletter
\renewcommand\subsection{
  \renewcommand{\sfdefault}{pag}
  \@startsection{subsection}%
  {2}{0pt}{.8\baselineskip}{.4\baselineskip}{\raggedright
    \sffamily\itshape\small\bfseries
  }}
\renewcommand\section{
  \renewcommand{\sfdefault}{phv}
  \@startsection{section} %
  {1}{0pt}{\baselineskip}{.8\baselineskip}{\centering
    \sffamily
    \scshape
    \bfseries
}}
\makeatother

\usepackage[left=1.02in,top=1.0in,right=1.02in,bottom=1.0in]{geometry}
\usepackage{multirow}

\setlist[enumerate]{leftmargin=0.8cm}
\setlist[itemize]{leftmargin=0.8cm}
\setlist[description]{leftmargin=0.0cm}

\newtheorem{thm}{Theorem}[section]

\newtheorem{lemma}[thm]{Lemma}

\newtheorem{que}{Question}
\newtheorem{conj}[que]{Conjecture}

\newtheorem{theorem}[thm]{Theorem}

\newtheorem{proposition}[thm]{Proposition}

\theoremstyle{definition}
\newtheorem{definition}[thm]{Definition}

\theoremstyle{remark}
\newtheorem{remark}[thm]{Remark}

\theoremstyle{remark}

\usepackage{color}

\newcommand{\rato}{\dashrightarrow}
\DeclarePairedDelimiter{\Set}{\lbrace}{\rbrace} %(utiliser \Set* pour ajuster la taille)
\def\bydef{\coloneqq}

\numberwithin{equation}{section}

\newcommand{\PP}{{\mathbb P}}

\newcommand{\C}{{\mathbb C}}

\newcommand{\Kbar}{{\overline{K}}}

\newcommand{\calC}{{\mathcal C}}

\newcommand{\cE}{{\mathcal E}}
\newcommand{\calF}{{\mathcal F}}

\newcommand{\calL}{{\mathcal L}}

\newcommand{\calO}{{\mathcal O}}
\newcommand{\cO}{{\mathcal O}}

\usepackage[mathscr]{euscript}

\DeclareMathOperator{\Spec}{Spec}

\numberwithin{equation}{section}
\numberwithin{table}{section}

\title{Simply Connectedness and Hyperbolicity}
\author{Carlo Gasbarri}
\address{Carlo Gasbarri\newline
IRMA, Université de Strasbourg, 7, rue René-Descartes,
67084 Strasbourg Cedex, France }
\email{gasbarri@math.unistra.fr}
\author{Erwan Rousseau}
\address{Erwan Rousseau \newline
Univ Brest, CNRS UMR 6205, Laboratoire de Mathematiques de Bretagne Atlantique, F-29200 Brest, France}
\email{erwan.rousseau@univ-brest.fr}
\author{Amos Turchet}
\address{Amos Turchet \newline
Dipartimento di Matematica e Fisica, Universit\'a degli studi Roma 3, L.go S. L. Murialdo 1, 00146 Roma, Italy }
\email{amos.turchet@uniroma3.it}
\author{Julie Tzu-Yueh Wang}
\address{Julie Tzu-Yueh Wang \newline Institute of Mathematics, Academia Sinica 
No.\ 1, Sec.\ 4, Roosevelt Road
Taipei 10617, Taiwan}
\email{jwang@math.sinica.edu.tw}


\begin{document}
\begin{abstract}
We generalize to arbitrary dimension our previous construction of simply connected weakly-special but not special varieties. We show that they satisfy the function field and complex analytic part of Campana's conjecture. Moreover, we give the first examples, in any dimension, of smooth simply connected nonisotrivial projective varieties of general type that satisfy the function field Lang's conjecture.
\end{abstract}
\maketitle
%\tableofcontents

\input{introduction}

\input{weakly}
%\newpage

\input{funfield.tex}

\bibliography{references}{}
\bibliographystyle{alpha}

\end{document}

%% file: introduction.tex
\section{introduction}\label{introduction}

The goal of this article is to discuss hyperbolicity properties of several simply connected varieties, of any dimension. The interplay between hyperbolicity, arithmetic, and geometric properties of algebraic variaties is among the fundamental problems in Diophantine Geometry and the focus of conjectures of Green, Griffiths, Demailly, Lang and Vojta (see e.g. \cite{De97}, \cite[Conjecture 3.7]{Lang91}, \cite{Lan86}, \cite[Conjecture 3.4.3]{vojta_lect}). For our aims, we can summarize the main predictions in the following Conjecture.

\begin{conj}\label{conj:main}
Let $X$ be a nonsingular quasi-projective variety defined over a number field $K$. If $X$ is of (log) general type there exists a proper closed subset $Z$ (the \emph{exceptional locus}), in a compactification $\overline{X}$ of $X$, such that: 
\begin{description}
    \item[{(Arithmetic)}] $X$ is arithmetically hyperbolic modulo $Z$, i.e. over any ring of $S$-integers the $S$-integral points of $X$ outside of $Z$ are finite;
    \item[{(Function Fields)}] $X_{\Kbar}$ is algebraic hyperbolic modulo $Z$, i.e. there exists an ample line bundle $\calL$ on $\overline{X}$ and a positive constant $\varepsilon >0$ such that, for every nonsingular projective curve $\calC$ and every morphism $\varphi: \calC \to \overline{X}$, such that $\varphi(\calC)$ is not contained in $(\overline{X} \smallsetminus X) \cup Z$, the following holds:
    \[
        \deg \varphi^* \calL \leq \varepsilon \left( 2g(\calC) - 2 + N_\varphi^{[1]}(\overline{X} \smallsetminus X)\right),
    \]
    where $N_\varphi^{[1]}(\overline{X})$ is the cardinality of the support of $\varphi^*(\overline{X} \smallsetminus X)$;
    \item[{(Complex Analysis)}] $X$ is Brody hyperbolic modulo $Z$, i.e. every holomorphic map $f: \C \to X_\C$ such that $f(\C)$ is not contained in $Z$ is constant.
\end{description}
\end{conj}

Campana proposed a more general framework using his notion of \emph{special} varieties introduced in \cite{Ca04}. A quasi-projective variety $X = \overline{X} \smallsetminus D$, where $\overline{X}$ is a nonsingular projective variety and $D$ is an snc divisor, both defined over a number field $K$, is special if for every rank 1 saturated coherent sheaf $\calF \subset \Omega^p_{\overline{X}_\Kbar}(\log D)$ one has $\kappa(\overline{X},\calF) < p$ (where $\kappa$ denotes the Itaka dimension). Equivalently $X$ is special if it has no fibration of general type in the sense of Campana (see \cite[Theorem 2.27]{Ca04}). In particular, (log) general type varieties are not special, but the class of non special varieties is strictly larger. After Campana, one expects the following generalizations of the above conjectures to be satisfied:

\begin{conj}\label{conj:Campana}
\begin{description}
Let $X$ be a nonsingular quasi-projective variety defined over a number field $K$. If $X$ is \emph{not} special then
    \item[(Arithmetic)] the set of integral points $X(\calO_L)$ is not Zariski dense for every finite extension $L \supset K$;
    \item[(Function Fields)] there is a dominant map $\pi: X \to Y$ (with $\dim Y > 0$), an ample line bundle $\calL$ on $\overline{Y}$, a positive constant $\varepsilon >0$, and a proper closed subset $Z \subset \overline{Y}$ such that, for every nonsingular projective curve $\calC$ and every morphism $\varphi: \calC \to \overline{X}$ such that $\pi(\varphi(\calC))$ is not contained in $(\overline{Y} \smallsetminus Y) \cup Z$ the following holds:
    \[
        \deg \varphi^*(\pi^* \calL) \leq \varepsilon \left( 2g(\calC) - 2 + N_\varphi^{[1]}(\overline{Y} \smallsetminus Y)\right);
    \]
    \item[(Complex Analysis)] there is no entire curve $\C \to X$ with Zariski dense image.
\end{description}
\end{conj}

\medskip

In this paper we discuss various examples where we can prove these conjectures for \emph{simply connected} varieties.

The search of simply connected examples, where the above mentioned conjectures can be proven true, is a central theme in Diophantine Geometry. In fact some of the skepticism around the claims in the arithmetic setting, can be traced to the lack of examples of smooth \emph{projective} simply connected varieties $X$ of general type of dimension at least 2, defined over a number field $K$, where the rational points are degenerate, i.e. where $X(L)$ is not Zariski dense for every finite extension $L \supset K$. In particular, to our knowledge, there is no single example of a smooth projective simply connected surface of general type that is pseudo-arithmetically hyperbolic, or even where the rational points can be proven to be degenerate. 

For quasi-projective varieties the situation is different: Corvaja and Zannier in \cite[Corollary 2, Theorem 3]{Corvaja:Zannier:2010} gave the first example of a \emph{quasi-projective} simply connected surface of log general type where the integral points, over any ring of $S$-integers, are not Zariski-dense. In \cite[Theorem 1.3, Proposition 1.4]{RTW2} we used a truncated version of the Ru-Vojta method to produce examples in arbitrary dimensions of quasi-projective simply connected varieties with a degenerate set of $S$-integral points. Apart from independent interest, these examples serve as the key input to construct simply connected special varieties that are not of log-general type but that are weakly special (see Definition \ref{def:wspecial}). This was achieved in dimension three in \cite[Section 4]{RTW}. In this article we generalize our construction and produce simply connected quasi-projective weakly special varieties $X$, of any dimension, fibered over a special variety $S$. Furthermore, in Theorem \ref{th:main_th_wspecial} we show that these examples satisfy the function field and complex analytic part of Campana's conjecture. In particular the varieties $X$ contradict the so-called Weakly Special Conjecture \cite[Conjecture 12]{RTW}.\medskip

%\subsection*{Weakly Special Varieties}
%In the first part of this paper we generalize the examples given by the last three authors in \cite{RTW} of weakly special varieties that are not special. The key ingredient is a generalization of a construction of Corvaja and Zannier in \cite[Corollary 2, Theorem 3]{Corvaja:Zannier:2010}, who gave the first example of a \emph{quasi-projective} simply-connected surface of log general type where the integral points, over any ring of $S$-integers, are not Zariski-dense. In \cite[Theorem 1.3, Proposition 1.4]{RTW2} we used a truncated version of the Ru-Vojta method to produce examples in any dimension of quasi-projective simply-connected varieties with a degenerate set of $S$-integral points. A part from independent interest, these examples serve as the key input to construct simply-connected weakly special but not special (in the sense of Campana) varieties. Generalizing \cite[Section 4]{RTW} we construct varieties $X$, of any dimension, fibered over a special variety $S$. Furthermore, in Theorem \ref{th:main_th_wspecial} we show that the orbifold base $(S,\Delta)$ of the fibration has several hyperbolic properties that contraddict the so-called Weakly Special Conjecture \cite[Conjecture 12]{RTW}

In another direction we present what we believe to be the first example of smooth \emph{projective} simply connected \emph{non-isotrivial} varieties of general type over the function field of a curve, for which we can prove Lang's conjecture. Previous works of Debarre and Bogomolov \cite{debarre} have shown how to construct simply connected nonsingular projective varieties with ample cotangent (and more generally with the same fundamental group as any given nonsingular projective variety, see \cite[Proposition 26]{debarre}). By work of Demailly and Kobayashi, these varieties are algebraically hyperbolic and Brody hyperbolic. In particular they yield examples of isotrivial varieties over the function field of a curve where Lang's conjecture hold (for an analogue in characteristic $p$ see \cite{KPS}).

In Theorem \ref{th:proj_nonisotrivial_main}, building on Debarre's construction, we present examples of non-isotrivial varieties over the function field of a curve where we can prove Lang's conjecture. The key idea is to construct a pencil, inside a simply connected variety with ample cotangent, given by sufficiently general hyperplane sections.\medskip

%project is to generalize the results of our previous paper \cite{RTW} to higher dimensions. In \cite{RTW} we dealt with two competing conjectures that aim to characterize algebraic varieties defined over a number field $k$ that have a potentially dense set of $k$-rational points. On one hand Campana conjectured that the class of these varieties is the class of \emph{special varieties}, introduced in \cite{Ca04}, while the Weak Specialness Conjecture (see \cite[Conjecture 1.2]{HT}) predicts that these should be the weakly special varieties, i.e. varieties that do not admit any \'etale cover that dominates a variety of general type. In \cite[Theorem 4.2]{RTW} we constructed examples of quasi-projective threefolds that are not special but weakly-special (see also \cite{BT,CP} for other constructions), and proved in \cite[Theorem 6.5]{RTW} that such examples possess properties that contradict function field and analytic analogues of the Weak-Specialness conjecture.

\subsection*{Acknowledgements}
We thank Ariyan Javanpeykar for comments on an early version of the paper.
ER was supported by Institut Universitaire de France and the ANR project \lq\lq FOLIAGE\rq\rq{}, ANR-16-CE40-0008. AT was partially supported by PRIN2017 ``Advances in Moduli Theory and Birational Classification'' and PRIN2020 ``Curves, Ricci flat Varieties and their Interactions'' and is a member of GNSAGA-INdAM. JW was supported in part by Taiwan's NSTC grant 110-2115-M-001-009-MY3.

%% file: weakly.tex
\section{Weakly Special but not special varieties}\label{sec:weakly}

We recall the following definitions.

\begin{definition}\label{def:wspecial}
	A smooth quasi-projective variety $X$ over a field $K$ is \emph{weakly special} if for every \'{e}tale morphism $u: X' \to X$ the variety $X'$ does not admit any dominant rational map $f' : X' \to Z'$ to a positive dimensional variety $Z'$ of log-general type.
\end{definition}

The aim of this section is to construct weakly special, quasi-projective varieties $X_m$, of any dimension greater or equal to 3, fibered over a variety $S$, such that the orbifold base (see \cite[Definition 2.4]{RTW}) of the map $X_m \to S$ has hyperbolic properties that contradict the Weakly Specialness Conejcture \cite[Conjecture 12]{RTW}. We use the analogue of the construction given in \cite[Theorem 4.2]{RTW} and, for the degeneracy result, the new theorem \cite[Theorem 2]{RTW2} (and its function field and complex analytic analogues \cite[Theorem 13 and Theorem 18]{RTW2}). \medskip

One fundamental feature of our construction is simply connectedness. In fact, for a simply-connected variety $X$, to be weakly special but not special it is sufficient the check the following two properties:
\begin{enumerate}
	\item If $X\to Z$ is a dominant morphism to a positive dimensional variety $Z$, then $Z$ is not of general type;
	\item There exists a general type fibration (in the sense of Campana) $f\colon X\to Y$ (and the variety $Y$, by part (1) is \emph{not} of general type).
\end{enumerate}

We begin with a generalization of \cite[Theorem 4.2]{RTW} where we adapt our previous construction to the more general setting of quasi-projective varieties of arbitrary dimension. We show that the construction yields a simply-connected and not special quasi-projective variety. However, to conclude that the output variety $X$ is weakly special we need to add the extra hypothesis that the base the variety $S$ is \emph{special}. 

\begin{theorem}\label{th:wspecial_construction}
Let $m$ be a positive integer, and let $T,S = \overline{S} \smallsetminus D$ be two quasi-projective varieties with fibrations \(f: T \to \PP^1\) and \(g: \overline{S} \to \PP^1\) such that:
	\begin{enumerate}
	  \item
	    $T$ is a smooth surface and the fibration \(f\colon T\to\PP^{1}\) has a single multiple fiber \(f^{-1}(0)\eqqcolon m\cdot T_0\), with \(T_0\) a smooth elliptic curve, and another (singular) simply-connected fiber;
	  \item
	     $S$ is a smooth quasi-projective variety and the fibration \(g\colon \overline{S}\to\PP^{1}\) has a simple normal crossing fiber \(\overline{S}_{0}\bydef g^{-1}(0)\) such that:
	    \begin{enumerate}
	      \item
		the variety \(S\) is not of log general type but the orbifold variety 
		\[
		  (\overline{S},D + \Delta)
		  \bydef
		  (\overline{S}, D + (1-1/m)\cdot \overline{S}_{0})
		\]
		is an orbifold of general type;
	      \item
		the complement of \(\overline{S}_{0}\) in \(S\) is simply connected.
	    \end{enumerate}
	\end{enumerate}
      
Furthermore, let $X$ be the normalization of the total spaces of the natural (orbifold) elliptic fibration with equidimensional fibers defined by \(f\) and \(g\), i.e. 
	\[
	  X
	  \bydef(
	  S\times_{\PP^1}T)^\nu
	  \stackrel{F}{\longrightarrow}
	  (S,\Delta).
	\]
Then $X$ is a simply connected, quasi-projective smooth variety that is not special. Moreover, if $S$ is special, then $X$ is weakly special.
\end{theorem}

      \begin{proof}
	We first notice that the orbifold base of the fibration \(F\) is indeed \(\Delta\):
	by construction, the only multiple fibers lie above \(\overline{S}_{0}\), and \(F^{\ast}\overline{S}_{0}=m\cdot(\overline{S}_{0}\times T_{0})\).
	Since by assumption \((\overline{S},D+\Delta)\) is a of general type, it follows immediately that \(X\) is not special.

	We now show that $X$ is simply-connected. Let \(D_0\bydef F^{-1}(\overline{S}_{0})\); then, the fibration \(F\colon(X\smallsetminus D_0)\to(S\smallsetminus \overline{S}_{0})\) is a fibration without multiple fibers 
	(since \(f\colon(T\smallsetminus T_0)\to \PP^1\smallsetminus\Set{0}\) has no multiple fibers),
	and with a simply-connected fiber. 
	This implies that \(F_*\colon\pi_1(X\smallsetminus D_0)\to \pi_1(S\smallsetminus \overline{S}_{0})\) is an isomorphism. 
	Using our assumption, the group \(\pi_1(X\smallsetminus D_0)\) is hence trivial.
	This implies that $X$ is simply-connected since the natural map \(\pi_1(X\smallsetminus D_0)\to \pi_1(X)\) is surjective.

	Finally, let us assume that $S$ is special. Since $X$ is simply-connected, it does not admit any non-trivial \'etale cover, therefore, to prove that \(X\) is weakly special, it suffices to show that no fibration \(h\colon X\to Z\) exists, with \(Z\) of general type and of positive dimension.
	Assume by contradiction that such an \(h\) exists. 
	Since, by construction, the fibers of $F$ are special, it follows from \cite[Theorem 2.7]{Ca04} that there exists a dominant map $S \to Z$. But since $S$ is special this contradicts the assumption that $Z$ is of general type.
      \end{proof}

We now use \cite[Theorem 2]{RTW2} and \cite[Proposition 1]{RTW2} to construct examples of simply-connected weakly special but not special quasi-projective varieties, in any dimension. This extends our previous construction in dimension 2 given in \cite[Example 4.3, 4.4]{RTW}.\medskip

\paragraph*{\textbf{Construction.}} Let $n \geq 2$ and let $r \geq 2n+1$ be two integers. Let $D_0,D_1,\dots,D_r$ be hypersurfaces in general position in $\PP^n$ and assume that $D_0$ is smooth and $D_1 + \dots + D_r$ has simple normal crossing singularities. Furthermore, we assume that the degree of the hypersurfaces $D_{i}$'s satisfy the following conditions:
\[
	\deg D_i = \begin{cases}
	d_i &\text{ for } 1\leq i\leq r-n-1\\
	1 &\text{ for } r-n \leq i \leq r \\
	d_0 = \sum_{j=1}^{r-n-1} d_j &\text{ for } i = 0
	\end{cases}
\]
We define $T_i = D_i \cap D_0$ and $T$ to be the union of the $T_i$'s. 

\begin{lemma}\label{lem:kzero}
In the above setting, let $\pi: \overline{S} \to \PP^n$ be the blow-up of $\PP^n$ along $T$ and let $\widetilde{D}_i\bydef \pi_*^{-1} D_i$ be the strict transform of $D_i$ in $\overline{S}$. Finally, we let $D = \widetilde{D}_{r-n} + \dots + \widetilde{D}_{r}$. Then, the pair $(\overline{S},D)$ is special.
\end{lemma}
\begin{proof}
We start by computing the log canonical bundle of $(\overline S, D)$ as follows:
\[
	K_{\overline{S}} + D = \left(\pi^* K_{\PP^n} + E\right) + \left( \pi^* (D_{r-n} + \dots + D_r)  - E' \right) = \pi^*(K_{\PP^n} + D_{r-n} + \dots + D_r) + (E - E'),
\] 
where $E$ is the total exceptional divisor and $E'$ is the sum of the exceptional divisors in the pullback of $D_{r-n} + \dots + D_{r}$. Then, since $\overline{S}$ is smooth, and $\pi(E - E')$ has codimension two in $\PP^n$, we get the isomorphism
\[
	H^0\left(\overline{S}, \pi^*( K_{\PP^n} + D_{r-n} + \dots + D_r)\right) \to H^0\left(\overline{S}, \pi^*(K_{\PP^n} + D_{r-n} + \dots + D_r) + (E - E')\right),
\]
and conclude that $\kappa(\overline{S} \smallsetminus D) = \kappa\left(\PP^n \smallsetminus (D_{r-n}+\dots+D_{r})\right)$.
Moreover, since $D_{r-n}+\dots+D_{r}$ are by definition $n+1$ hyperplanes in $\PP^n$, we get that $\kappa(\overline{S}\smallsetminus D) = 0$. The conclusion follows by \cite[Th\'eor\`eme 6.7]{Ca11}.
\end{proof}

We are now in the position to apply Theorem \ref{th:wspecial_construction}. The fibration $f: T \to \PP^1$ is the same fibration constructed in \cite[Example 4.3]{RTW}. The second fibration is obtained as follows: consider the fibration induced by $D_0$ and $D_1+\dots+D_{r-n-1}$ in $\PP^n$. This induces a rational map $g': \PP^n \rato \PP^1$ which yields a fibration $g: \overline S \to \PP^1$. We claim that $g$ satisfies the hypothesis of Theorem \ref{th:wspecial_construction}. First of all, by construction, $S$ is smooth and the fiber $\overline{S}_0 = g^{-1}(0)$, being the strict transform of $D_1 + \dots + D_{r-n-1}$ has simple normal crossing singularities. By \cite[Proposition 1]{RTW2} the variety $\overline{S} \smallsetminus (\widetilde{D}_1 + \dots + \widetilde{D}_{r-n-1} + D) = \overline{S} \smallsetminus (\widetilde{D}_1 + \dots + \widetilde{D}_r)$ is simply-connected.
Finally we have the following lemma.

\begin{lemma}\label{lem:orb_general_type}
In the above setting, let $m \geq 2$ be an integer and let 
\[
	\Delta = \left(1 - \dfrac{1}{m}\right) \overline{S}_0 = \left(1 - \dfrac{1}{m} \right) \left( \widetilde{D}_1 + \dots + \widetilde{D}_{r - n - 1} \right).
\]
Then the orbifold $(\overline{S},D + \Delta)$ is of general type.
\end{lemma}
\begin{proof}
We compute the orbifold canonical bundle following Lemma \ref{lem:kzero}. Recall that we denoted by $E$ the total exceptional divisor of the blow up $\pi: \overline{S} \to \PP^2$ along $T$, and by $E'$ the exceptional divisors in the pullback of $D_{r-n} + \dots + D_{r}$. Therefore we have
\begin{align*}
	K_{\overline{S}} + D + \Delta  &= \left(\pi^* K_{\PP^n} + E\right) + \left( \pi^* (D_{r-n} + \dots + D_r)  - E' \right) + \\
	&\qquad \qquad + \left( 1 - \dfrac{1}{m}\right)\left(\pi^* (D_1 + \dots + D_{r-n-1}) - (E - E')\right) \\
	&= \pi^*(K_{\PP^n} + \left( 1 - \dfrac{1}{m}\right)(D_{1} + \dots + D_{r-n-1}) + D_{r-n} + \dots + D_r) + \dfrac{1}{m} (E-E')
\end{align*}	
The conclusion holds since
\[
	\deg (K_{\PP^n} + \left( 1 - \dfrac{1}{m}\right)(D_{1} + \dots + D_{r-n-1}) + D_{r-n} + \dots + D_r) = -n -1 + \left( 1 - \dfrac{1}{m}\right)d_0 + n + 1 > 0,
	\]
 $\pi$ is a generically finite map, and $E - E'$ is effective.

\end{proof}

To summarize, we have therefore constructed quasi-projective varieties $(X,D)$ that are simply-connected, weakly special but not special. Following our previous paper \cite{RTW} we can now show that these varieties, at least for high enough $m$, present behaviour that contradict the function field and analytic analogue of the \emph{Weak Special Conjecture}, see \cite[Conjecture 12]{RTW}.

We start by a pair $(X,D)$ constructed above, so in particular we have a fibration $F: X \to S = \overline{S} \smallsetminus D$ with orbifold base $\Delta$. Then we have the following result which is a combination of \cite[Theorem 7.4]{RTW2} (together with \cite[Theorem 5.10 and 5.11]{RTW}) for the analytic part, and \cite[Theorem 8.2]{RTW2} together with \cite[Lemma 5.9]{RTW2} for the function field part (we refer to the aforemoentioned papers for definitions and notations).

\begin{proposition}\label{prop:weak_hyp}
	Let $\overline{S},D$, and $\Delta$ as above. Then:
	\begin{enumerate}
		\item for every $\varepsilon > 0$ there exists a positive integer $Q > 0$ and a proper algebraic closed subset $Z$ such that: for every holomorphic map $f: \C \to \overline{S}$ such that $f(\C)$ is not contained in $Z$, the following holds:
		\[
		T_{\calL,f}(r) \leq_{\text{exc}} \varepsilon \sum_{i=1}^r N_f^{[Q]}(\widetilde{D}_i,r).
		\]
		In particular, for every large $m$, the \emph{orbifold} $(\overline{S}, D + \Delta)$ is Brody hyperbolic modulo a proper subset;
		\item there exist an ample line bundle $\calL$ on $\overline{S}$, a positive constant $A>0$ and a proper closed subset $Z \subset \overline{S}$ such that:
		 for every non-singular projective curve $\calC$ and every morphism $\varphi: \calC \to \overline{S}$ such that $\varphi(\calC)$ is not contained in $Z$, the followign holds:
		\[
			\deg \varphi^*\calL \leq A \left( 2g(\calC) - 2 + N^{[1]}_\varphi(D + \lceil \Delta \rceil) \right).
		\]
		In particular, for every large $m$, the \emph{orbifold} $(\overline{S}, D + \Delta)$ is algebraically hyperbolic modulo a proper subset.
	\end{enumerate}
\end{proposition}

Finally, combining Proposition \ref{prop:weak_hyp} with the techniques developed in \cite{RTW} we obtain the following main result of this section.

\begin{theorem}\label{th:main_th_wspecial}
	For every $n \geq 3$ there exists a smooth quasi-projective variety $X$ of dimension $n$ which is simply-connected,  weakly special but not special, together with a fibration $F: X \to (S:=\overline{S} \setminus D, \Delta)$
	%Moreover, there exists a positive integer $m_0$ with the following property: if the orbifold base of $F$ is $\Delta = \sum_i(1 - 1/m_i)D_i'$ and all the multiplicities $m_i$ satisfy $m_i \geq m_0$ we have
	Moreover the following holds:
	\begin{enumerate}
		\item every entire curve $\C \to X$ is algebraically degenerate;
		\item there exist a positive constant $A > 0$ and a proper algebraic closed subset $Z \subset S$, such that for every morphism $\varphi: \calC \to \overline{X}$, such that $F(\varphi(\calC))$ is not contained in $Z$, the following holds:
		%(S,\Delta)$ is pseudo-algebraically hyperbolic.
		\[
			\deg F(\varphi(\calC)) \leq A \left( 2g(\calC) - 2 + N^{[1]}_{\varphi \circ F}(D) \right)
		\]
	\end{enumerate}
\end{theorem}
\begin{proof}
	The construction of the variety $X$ and of the fibration $F:X \to S$ was given above. In particular, in the above notation, $D = (\widetilde{D}_{r-n} + \dots + \widetilde{D}_{r})$.

	The proof of (1) follows the same steps of the proof of \cite[Theorem 5.11]{RTW}: part (1) of Proposition \ref{prop:weak_hyp} gives the analogue of equation (5.9) in op.cit. and then one concludes for sufficiently large multiplicities $m$.

	The proof of (2) is obtained by combining part (2) of Proposition \ref{prop:weak_hyp} with \cite[Lemma 5.9]{RTW}.
\end{proof}

%% file: funfield.tex
\section{A projective simply connected example}
In this section we let $K$ be an algebraically closed field of characteristic zero and $F$ be the field $K(t)$. We recall Lang's Conjecture for varieties over $F$.

\begin{conj}[{Lang, see \cite[3.16]{Lang91}}]\label{conj:LangFF}
   Let $X$ be a smooth projective variety $X$ defined over $F$, non birational to an isotrivial variety. Assume that $X$ is of general type, then for every finite extension $L \supset F$, the set of $L$-rational points $X_F(L)$ is not Zariski dense. 
\end{conj}

In this section we will prove the following Theorem:

\begin{theorem}\label{th:proj_nonisotrivial_main}  For every positive integer $n>1$, there exists a smooth projective {\rm simply connected} variety $X_F$ of general type, of dimension $n$, which verifies Conjecture \ref{conj:LangFF}.
\end{theorem}

A variety $X_F$ over $F$ is said to be simply connected if the natural morphism of groups \linebreak $\pi_1(X_F,x_o)\to Gal(\overline{F}/F)$ is an isomorphism ($\overline F$ being an algebraic closure of $F$). This is equivalent to say that the $\overline F$ variety $X_F\otimes_F\overline F$ is simply connected.

\begin{remark}
   \begin{itemize}
      \item The same strategy of the proof gives examples of smooth projective varieties over $F$ non birational to an isotrivial variety with the fundamental group of any projective-variety (as in \cite[Proposition 26]{debarre}) which verify Conjecture \ref{conj:LangFF}.
      \item The proof of Theorem \ref{th:proj_nonisotrivial_main} shows that a stronger statement holds in these examples, namely the existence of an \emph{exceptional set}, i.e. a proper closed set, $T_F$, such that $X_F(L)\smallsetminus T_F(L)$ is finite. As pointed out in \cite[Remark 3.3]{RTW}, one cannot expect such a strong statement to hold for every general type variety defined over a function field, assuming only that the variety is not birational to an isotrivial variety.
   \end{itemize}
\end{remark}

We recall the following facts on which the proof of the theorem will relay on:

If $\cE$ is a vector bundle on a variety $Z$, we will denote by $\pi_\cE:\PP(\cE)\to Z$ the natural projection. 

\begin{definition} A vector bundle $\cE$ on a projective variety $Z$ is said to be ample if the tautological line bundle $\cO(1)$ on $\PP(\cE)$ is ample. \end{definition}

\begin{proposition}[{See \cite[Corollary 6.3.30]{Laza}}]\label{finitehom}Let $Z$ be a smooth projective variety with ample cotangent bundle. Then, for every smooth projective curve $Y$, the set $Hom_\ast(Y,Z)$ of non constant maps from $Y$ to $Z$ is finite.\end{proposition}

We can  generalize the proposition above to:

\begin{proposition}\label{finitehom2} Let $Z$ be a projective variety, birational to a variety $Z'$ which is projective smooth and with ample cotangent bundle. Then there is a proper closed subset $B\varsubsetneq Z$ for which the following property holds: For every smooth projective curve $Y$, the set $Hom_B(Y,Z)$ of non constant maps such that $f(Y)\not\subset B$, is finite.\end{proposition}

\begin{proof} We can find a diagram of birational morphisms

\begin{equation}
\xymatrix{&Z''\ar[ld]_p\ar[rd]^q&\\
Z& &Z'\\}
\end{equation}
Let $E_p$ and $E_q$ be the exceptional loci of $p$ and $q$ respectively. The closed subset 
 $B=p(E_p\cup E_q)$ does not coincide with $Z$.  Suppose that $f: Y\to Z$ is a map whose image is not contained in $B$. Then, since $Y$ is smooth, $f$ lifts uniquely to a map $f'':Y\to Z''$ and to a map $f':=q\circ f':Y\to Z'$. Consequently we find an injective map $Hom_B(Y,Z)\hookrightarrow Hom_\ast(Y,Z')$. The conclusion follows from Proposition \ref{finitehom}.
\end{proof}

We need also the following fact:

\begin{proposition} There exist smooth projective varieties $Z$ which are simply connected and with ample cotangent bundle.\end{proposition}
This is a particular case of Proposition 26 in \cite{debarre}.

Now we can prove Theorem \ref{th:proj_nonisotrivial_main}:
\begin{proof}[Proof of Theorem \ref{th:proj_nonisotrivial_main}] Let $Z$ be a simply connected smooth projective variety of dimension $n+1$, with ample cotangent bundle. Let $L$ be a very ample line bundle on $Z$. Let $H_1$ and $H_2$ be two smooth global sections of $L$ intersecting transversally on a smooth subvariety $B$ of codimension two. 

Let $\tilde Z\to Z$ be the blow up of $Z$ along $B$. The two sections define a morphism $\pi :\tilde Z\to \PP^1$. The cartesian product 
\begin{equation}
\xymatrix{X_F \ar[r] \ar[d] & \tilde Z \ar[d]^\pi\\
\Spec(F) \ar[r]^\eta &\PP^1\\}
\end{equation}
defines a variety $X_F$ over $F$. 

For every closed point $t\in\PP^1$ the variety $X_t:=\pi^{-1}(t)$, if smooth, is simply connected because of Lefschetz hyperplane Theorem and the hypothesis that $n>1$. Consequently $X_F$ is simply connected. 

Let $L/F$ is a finite extension and let $B_L\to \PP^1$ be the finite covering associated to it. Each point $p\in X_F(L)$ extends, by the evaluative criterion of properness, to a morpism $P:B_F\to Z'$. 

By Proposition \ref{finitehom2}, if $P(B_F)$ is not contained in the exceptional divisor of $\tilde Z$, it belongs to a finite list. 
Consequently, there is a closed set $T_F\subsetneq X_F$ such that, $X_F(L)\setminus T_F(L)$ is finite. 

The variety $X_F$ is not birationally equivalent to an an isotrivial variety, otherwise, we could find a finite extension $L$ for which $X_F(L)$ is Zariski dense. \end{proof}

%% file: SimplyConnect.bbl
\begin{thebibliography}{RTW23}

\bibitem[Cam04]{Ca04}
Fr\'{e}d\'{e}ric Campana.
\newblock Orbifolds, special varieties and classification theory.
\newblock {\em Ann. Inst. Fourier (Grenoble)}, 54(3):499--630, 2004.

\bibitem[Cam11]{Ca11}
Fr\'{e}d\'{e}ric Campana.
\newblock Orbifoldes g\'{e}om\'{e}triques sp\'{e}ciales et classification
  bim\'{e}romorphe des vari\'{e}t\'{e}s k\"{a}hl\'{e}riennes compactes.
\newblock {\em J. Inst. Math. Jussieu}, 10(4):809--934, 2011.

\bibitem[CZ10]{Corvaja:Zannier:2010}
Pietro Corvaja and Umberto Zannier.
\newblock Integral points, divisibility between values of polynomials and
  entire curves on surfaces.
\newblock {\em Adv. Math.}, 225(2):1095--1118, 2010.

\bibitem[Deb05]{debarre}
Olivier Debarre.
\newblock Varieties with ample cotangent bundle.
\newblock {\em Compos. Math.}, 141(6):1445--1459, 2005.

\bibitem[Dem97]{De97}
Jean-Pierre Demailly.
\newblock Algebraic criteria for {K}obayashi hyperbolic projective varieties
  and jet differentials.
\newblock In {\em Algebraic geometry---{S}anta {C}ruz 1995}, volume~62 of {\em
  Proc. Sympos. Pure Math.}, pages 285--360. Amer. Math. Soc., Providence, RI,
  1997.

\bibitem[KPS22]{KPS}
Stefan Kebekus, Jorge~Vit\'{o}rio Pereira, and Arne Smeets.
\newblock Failure of the {B}rauer-{M}anin principle for a simply connected
  fourfold over a global function field, via orbifold {M}ordell.
\newblock {\em Duke Math. J.}, 171(17):3515--3591, 2022.

\bibitem[Lan86]{Lan86}
Serge Lang.
\newblock Hyperbolic and {D}iophantine analysis.
\newblock {\em Bull. Amer. Math. Soc. (N.S.)}, 14(2):159--205, 1986.

\bibitem[Lan91]{Lang91}
Serge Lang.
\newblock {\em Number theory. {III} - Diophantine geometry}, volume~60 of {\em
  Encyclopaedia of Mathematical Sciences}.
\newblock Springer-Verlag, Berlin, 1991.

\bibitem[Laz04]{Laza}
Robert Lazarsfeld.
\newblock {\em Positivity in algebraic geometry. {I}}, volume~48 of {\em
  Ergebnisse der Mathematik und ihrer Grenzgebiete (3)}.
\newblock Springer-Verlag, Berlin, 2004.

\bibitem[RTW21]{RTW}
Erwan Rousseau, Amos Turchet, and Julie Tzu-Yueh Wang.
\newblock Nonspecial varieties and generalised {L}ang–{V}ojta conjectures.
\newblock {\em Forum of Mathematics, Sigma}, 9:1--29, 2021.

\bibitem[RTW23]{RTW2}
Erwan Rousseau, Amos Turchet, and Julie Tzu-Yueh Wang.
\newblock Divisibility of polynomials and degeneracy of integral points.
\newblock {\em Mathematische Annalen}, 2023.

\bibitem[Voj87]{vojta_lect}
Paul Vojta.
\newblock {\em {Diophantine Approximations and Value Distribution Theory}},
  volume 1239 of {\em Lecture Notes in Mathematics}.
\newblock Springer Berlin Heidelberg, 1987.

\end{thebibliography}
